\documentclass[notitlepage,11pt,reqno]{amsart}
\usepackage{amssymb,amsmath,mathrsfs,amsthm}
\usepackage[usenames]{color}
\usepackage[breaklinks=true]{hyperref}
\usepackage[letterpaper]{geometry}
\newcommand{\bb}{\ensuremath{\mathbb B}}
\newcommand{\be}{\ensuremath{\mathbb E}}
\newcommand{\N}{\ensuremath{\mathbb N}}
\newcommand{\Z}{\ensuremath{\mathbb Z}}
\newcommand{\Q}{\ensuremath{\mathbb Q}}
\newcommand{\R}{\ensuremath{\mathbb R}}
\newcommand{\G}{\ensuremath{\mathscr G}}
\newcommand{\scrc}{\ensuremath{\mathscr C}}

\newcommand{\bp}[1]{\ensuremath{\mathbb P}_\mu \left( #1 \right)}
\newcommand{\comment}[1]{}
\DeclareMathOperator{\Hom}{Hom}

\newtheorem{theorem}[equation]{Theorem}
\newtheorem{prop}[equation]{Proposition}
\newtheorem{cor}[equation]{Corollary}
\newtheorem{lemma}[equation]{Lemma}

\theoremstyle{definition}
\newtheorem{remark}[equation]{Remark}
\newtheorem{definition}[equation]{Definition}
\newtheorem{example}[equation]{Example}

\numberwithin{equation}{section}

\begin{document}
\vspace*{-1mm}

\title[Khinchin--Kahane, L\'evy inequalities; transfer from normed groups
to Banach spaces]{The Khinchin--Kahane and L\'evy inequalities for\\
abelian metric groups, and transfer from\\
normed (abelian semi)groups to Banach spaces}

\author{Apoorva Khare}
\address[A.~Khare]{Indian Institute of Science;
Analysis and Probability Research Group; Bangalore, India}
\email{khare@math.iisc.ac.in}

\author{Bala Rajaratnam}
\address[B.~Rajaratnam]{University of California, Davis, USA;\hfill\break
\hspace*{29.2mm}University of Sydney, Sydney, Australia (visiting)}
\email{brajaratnam01@gmail.com}

\thanks{Partially supported by US Air Force Office of Scientific Research
grant FA9550-13-1-0043, US National Science Foundation under grant
DMS-0906392, DMS-CMG 1025465, AGS-1003823, DMS-1106642,
DMS-CAREER-1352656, Defense Advanced Research Projects Agency DARPA YFA
N66001-111-4131, the UPS Foundation, SMC-DBNKY, a Young Investigator
Award (Infosys Foundation),
Ramanujan Fellowship grant SB/S2/RJN-121/2017, MATRICS grant
MTR/2017/000295, SwarnaJayanti Fellowship grants SB/SJF/2019-20/14 and
DST/SJF/MS/2019/3 from SERB and DST (Govt.~of India), and by grant
F.510/25/CAS-II/2018(SAP-I) from UGC (Govt.~of India).}



\date{\today}

\begin{abstract}
The Khinchin--Kahane inequality is a fundamental result in the
probability literature, with the most general version to date holding in
Banach spaces. Motivated by modern settings and applications, we
generalize this inequality to arbitrary metric groups which are abelian.

If instead of abelian one assumes the group's metric to be a norm (i.e.,
$\mathbb{Z}_{>0}$-homogeneous), then we explain how the inequality
improves to the same one as in Banach spaces.
This occurs via a ``transfer principle'' that helps carry over questions
involving normed metric groups and abelian normed semigroups into the
Banach space framework. 
This principle also extends the notion of the expectation to random
variables with values in arbitrary abelian normed metric semigroups
$\mathscr{G}$. We provide additional applications, including studying
weakly $\ell_p$ $\mathscr{G}$-valued sequences and related Rademacher
series.

On a related note, we also formulate a ``general'' L\'evy inequality,
with two features:
(i)~It subsumes several known variants in the Banach space literature;
and
(ii)~We show the inequality in the minimal framework required to state
it: abelian metric groups.
\end{abstract}
\maketitle

\vspace*{-2mm}

%{{{1 Section 1 - Introduction
\section{Introduction}

The Khinchin--Kahane inequality is a classical inequality in the
probability literature. It was initially studied by Khinchin \cite{Khi}
in the real case, and later extended by Kahane \cite{Kah2} to normed
linear spaces. A detailed history of the inequality can be found in
\cite{LO}. We begin by presenting a general version of the inequality for
Banach spaces, as well as a sharp constant in some cases.

\begin{definition}
A \textit{Rademacher variable} is a Bernoulli variable that equals $\pm
1$ with probability $\frac{1}{2}$.
\end{definition}

\begin{theorem}[Kahane \cite{Kah2}; Lata{\l}a and Oleszkiewicz \cite{LO}]
For all $p,q \in [1,\infty)$, there exists a universal constant $C_{p,q}
> 0$ depending only on $p,q$, such that for all choices of Banach spaces
$\bb$, finite sets of vectors $x_1, \dots, x_n \in \bb$, and independent
Rademacher variables $r_1,\dots,r_n$,
\begin{equation}\label{Ekhinchin-Banach}
\be \left[ \left\| \sum_{k=1}^n r_k x_k \right\|^q \right]^{1/q}
\leqslant C_{p,q} \cdot \be \left[ \left\| \sum_{k=1}^n r_k x_k
\right\|^p \right]^{1/p}.
\end{equation}

\noindent If moreover $p=1 \leqslant q \leqslant 2$, then 
the constant $C_{1,q} = 2^{1-1/q}$ is optimal.
\end{theorem}

\noindent On this note, see also the work of Kalton
\cite[Section~6]{Kalton}, involving (quasi-)Banach spaces with specified
type.

Notice that to state the theorem, one only requires Rademacher -- i.e.,
random symmetric -- sums of vectors (see \cite[\S 2.5]{Kah} for more on
Rademacher series). Thus, it is possible to state the result more
generally than in a normed linear space: in fact, in any abelian group
$\G$ equipped with a translation-invariant metric. Now it is natural to
ask whether a variant of the Khinchin--Kahane inequality holds in this
general (and strictly larger) setting. One of our main results provides a
positive answer to this question; see Theorem~\ref{Tkhinchin}.

In working outside Banach spaces, we are motivated by several reasons,
both classical and modern. Traditionally, probability theory is now
well-established in the Banach space setting; see the classic
text~\cite{LT}. In greater generality, Fourier analysis and Haar measure
for locally compact abelian groups, and metric group-valued random
variables have been well-studied, see e.g.~\cite{Gre,Ru}.
In this vein, it is of interest to prove stochastic inequalities in the
most primitive framework required to state them.
In~\cite{KR4}, we showed such a extension of the Hoffmann-J{\o}rgensen
inequality for all metric semigroups, followed by applications and other
extensions in~\cite{KR3}. The present paper is in a parallel vein,
showing an extension of the Khinchin--Kahane inequality to abelian metric
groups.

There are also modern reasons to work with metric groups. In modern-day
settings, one often studies random variables with values in permutation
groups, or more generally, abelian/compact Lie groups such as
lattices/tori. Moreover, data can be manifold-valued, living in e.g.~real
or complex Lie groups rather than in linear spaces. Other frameworks
arise from the study of large networks, e.g.\ the space of graphons with
the cut-norm \cite{Lo}, or the space of labelled graphs \cite{KR1}. The
latter is a $2$-torsion group, so cannot embed into a Banach space. Thus
there is renewed modern interest in studying probability outside the
Banach space paradigm. Our work lies firmly in this setting.

We now state our novel contributions. The first -- see above -- extends
the Khinchin--Kahane inequality to abelian metric groups
(Theorem~\ref{Tkhinchin} below).
Second, in the course of proving Theorem~\ref{Tkhinchin}, we also provide
a two-fold extension of L\'evy's inequality; see Theorem~\ref{Llevyineq}.
In keeping with the above philosophy, this unifies at once several
existing variants of the inequality in the literature, which to our
knowledge had not been consolidated within a common framework. Moreover,
we show the result in the minimal framework required to state it: for all
abelian metric groups.

In addition to these two results for metric \textit{groups}, we write
down a useful ``transfer principle'' in Theorem~\ref{Tnormed}, which
holds more generally: for normed metric \textit{semigroups}. Note,
probability over metric (semi)groups $\G$ has a fundamental distinction
from Banach spaces: the unavailability of an expectation function. It is
thus natural to seek classes of metric semigroups $\G$ for which the
expectation makes sense for every $\G$-valued random variable.
Theorem~\ref{Tnormed} provides such a class: where $\G$ is in fact
``normed'' (defined below) and abelian -- or merely normed if $\G$ is a
group. In this case the expectation does not necessarily live in $\G$,
but in its ``Banach space closure'' (below).

Theorem~\ref{Tnormed} enables defining linear functionals, operator
spaces, and dual spaces over all abelian normed semigroups $\G$, and
therefore extends the powerful theory of functional analysis to all such
semigroups. Furthermore, we provide several applications of
Theorem~\ref{Tnormed}, including extending the notion of weakly
$\ell_p$-sequences to $\G$-valued random variables; as a consequence,
several results in the probability literature, including those of
Dilworth and Montgomery-Smith \cite{DM} (as well as prior results of
Talagrand) extend to abelian normed semigroups and to all normed
groups. See Section~\ref{Snormed} for these applications of our results;
the assertions about normed groups make use of a fact on bi-invariant
metrics from geometry, which arose from the present body of work, and has
been resolved in a recent Polymath project \cite{Po}.
%}}}

\section{Khinchin--Kahane inequality and L\'evy inequality for abelian
metric groups}\label{Skhinchin}

We begin by extending the Khinchin--Kahane inequality from Banach spaces
to arbitrary abelian metric groups. We also prove a general version of
L\'evy's inequality, for abelian metric groups.

%{{{1 Section 2.1 - Metric semigroups
\subsection{Metric semigroups}\label{Ssemigroup}

Before the main results, we study basic properties of metric
(semi)groups.

\begin{definition}\label{Dsemi}
A \textit{metric semigroup/monoid/group} is defined to be a
semigroup/monoid/group $(\G, \cdot)$ equipped with a metric $d : \G
\times \G \to [0,\infty)$ that is translation-invariant. In other words,
\[
d(ac,bc) = d(a,b) = d(ca,cb), \qquad \forall a,b,c \in \G.
\]
\end{definition}

To begin, note that for such a semigroup $\G$ -- for instance if $\G$ is
abelian -- the following ``triangle inequality'' is straightforward, and
used below without further reference:
\begin{equation}\label{Etriangle}
d(y_1 y_2, z_1 z_2) \leqslant d(y_1, z_1) + d(y_2, z_2), \quad \forall
y_k, z_k \in \G.
\end{equation}

We will also require the following preliminary result and its corollary.

\begin{prop}\label{Pstrict}
Suppose $(\G, \cdot, d)$ is a metric semigroup, and $a,b \in \G$. Then
\begin{equation}\label{Estrict}
d(a,ba) = d(b,b^2) = d(a,ab)
\end{equation}

\noindent is independent of $a \in \G$. Moreover, $\G$ has at most one
idempotent (i.e., $b \in \G$ such that $b^2 = b$). Such an element $b$ is
automatically the unique two-sided identity in $\G$, making it a metric
monoid.
\end{prop}

\begin{proof}
Equation~\eqref{Estrict} is immediate using the translation-invariance of
$d$:
\[
d(a,ba) = d(ba, b^2 a) = d(b, b^2) = d(ab, a b^2) =
d(a, ab).
\]

\noindent Next, if $\G$ has idempotents $b,b'$, then using
Equation~\eqref{Estrict},
\[
d(b,b') = d(b^2, bb') = d(b^2, b^2 b') = d(b,bb') = d(b', (b')^2) = 0.
\]

\noindent Hence $b = b'$. Moreover, given such an idempotent $b \in \G$,
compute using Equation~\eqref{Estrict}:
\[
d(a,ba) = d(a,ab) = d(b,b^2) = 0, \quad \forall a \in \G.
\]

\noindent Hence $b$ is automatically the unique two-sided identity in
$\G$.
\end{proof}

\begin{cor}\label{Cstrict}
A set $\G$ is a metric semigroup only if $\G$ is either a metric monoid,
or the set of non-identity elements in a metric monoid $\G'$. This is if
and only if the number of idempotents in $\G$ is one or zero,
respectively. Moreover, the metric monoid $\G'$ is (up to a monoid
isomorphism) the unique smallest element in the class of metric monoids
containing $\G$ as a sub-semigroup.
A finite metric semigroup is a metric group.
\end{cor}

\begin{proof}
Given a semigroup $\G$ that is not already a monoid, to adjoin an
``identity'' element $1'$ one defines: $1' \cdot a = a \cdot 1' := a\
\forall a \in \G' := \G \sqcup \{ 1' \}$. Also extend $d_\G$ to $d_{\G'}
: \G' \times \G' \to [0,\infty)$ via: $d_{\G'}(1',1') = 0$ and
$d_{\G'}(1',b) = d_{\G'}(b,1') := d_\G(b,b^2)$ for $b \in \G$. Then $\G'$
is a metric monoid.
The next assertion follows from Proposition~\ref{Pstrict}. It is
clear that the monoid $\G' \supseteq \G$ is uniquely determined.
The final part holds since left- and right- multiplication by any $a \in
\G$ are bijections.
\end{proof}

\begin{remark}\label{Rstrict}
We will denote the unique metric monoid containing a given metric
semigroup $\G$ by $\G' := \G \cup \{ 1' \}$. Note that the idempotent
$1'$ may already be in $\G$, in which case $\G = \G'$.
One consequence of Corollary~\ref{Cstrict} is that instead of working
with metric semigroups, one can use the associated monoid $\G'$ instead.
(In other words, the (non)existence of the identity is not an issue in
such cases.)
This helps simplify other calculations. For instance, what would usually
be a lengthy, inductive computation now becomes much simpler: 
by the triangle inequality~\eqref{Etriangle}, for all $k,l \geqslant 0$,
\[
d_\G(z_0 \cdots z_k, z_0 \cdots z_{k+l}) = d_{\G'}(1', \prod_{i=1}^l
z_{k+i}) \leqslant \sum_{i=1}^l d_{\G'}(1', z_{k+i}) = \sum_{i=1}^l
d_\G(z_0, z_0 z_{k+i}),
\forall z_0, \dots, z_{k+l} \in \G.
\]
\end{remark}
%}}}

%{{{1 Section 2.2 - Khinchin--Kahane and L\'evy inequalities
\subsection{Khinchin--Kahane and L\'evy inequalities}

In the sequel, we will deal mostly with metric (semi)groups $\G$ that are
abelian. Thus we mostly use additive notation -- and only where we do
not, should the reader assume that $\G$ need not \textit{a priori} be
abelian.

Our first main result is a Khinchin--Kahane (type) inequality for
arbitrary abelian metric groups.

\begin{theorem}[Khinchin--Kahane]\label{Tkhinchin}
Given $q \in [1,\infty)$, there exists a universal constant $K_q > 0$
depending only on $q$ such that for all choices of abelian metric groups
$\G$, finite sequences of elements $x_1, \dots, x_n \in \G$ (for any
$n>0$), independent Rademacher variables $r_1, \dots, r_n$,
and scalar $p \in [1,\infty)$,
\begin{equation}\label{Ekhinchin}
\be_\mu \left[ d \left( 0, 2^l \sum_{k=1}^n r_k x_k \right)^q
\right]^{1/q} \leqslant K_q \cdot \be_\mu \left[ d \left( 0,
\sum_{k=1}^n r_k x_k \right)^p \right]^{1/p},
\end{equation}

\noindent where $l \in \Z_{>0}$ is such that $2^{l-1} \leqslant q < 2^l$.
In fact we may choose $K_q = 64 q^2 (q/4)^{1/q}$.
\end{theorem}

Existing variants in the literature fall under the special case where $\G
= \bb$ is a Banach space~\eqref{Ekhinchin-Banach}. Note that the
inequality~\eqref{Ekhinchin} in this more general case does not compare
the same quantities as the classical Khinchin--Kahane
inequality~\eqref{Ekhinchin-Banach} does, and to the best of our
knowledge, is a novel inequality that does not follow from the Banach
space version~\eqref{Ekhinchin-Banach}. (In Section~\ref{Snormed}, we
will see how a ``norm'' on $\G$ updates both the expressions
in~\eqref{Ekhinchin} and the constant $K_q$ to yield the Banach-space
counterparts.)
Also remark for completeness that a separability assumption on $\G$ is
not required, since one may restrict to the subgroup generated by $x_1,
\dots, x_n$.

Theorem~\ref{Tkhinchin} provides an example of ``universal constants''
which help compare $L^p$-norms of sums of independent $\G$-valued
variables, across various $p > 0$. This theme is explored in greater
detail and generality for abelian metric semigroups in related
work~\cite{KR3}.

The proof of Theorem~\ref{Tkhinchin} uses L\'evy's inequality for abelian
metric groups $\G$. To this end, we first define symmetric random
variables and show a general version of L\'evy's inequality over $\G$.

\begin{definition}
If $(\G, +, 0, d)$ is an abelian metric group and $I$ is a totally
ordered finite set, then a tuple $(X_i)_{i \in I}$ of random variables in
$L^0(\Omega,\G)$ is \textit{symmetric} if for all subsets $J \subseteq I$
and all functions $\varepsilon : J \to \{ \pm 1 \}$, the $2^{|J|}$
ordered tuples of variables $(\varepsilon(j) X_j)_{j \in J}$ all have the
same joint distribution -- i.e., this is independent of $\varepsilon$.
\end{definition}

\begin{theorem}[L\'evy's inequality]\label{Llevyineq}
Fix an abelian metric group $(\G, +, 0, d)$, integers $m,n \in \Z_{>0}$,
and symmetric random variables $X_1, \dots, X_n \in L^0(\Omega,\G)$. Also
fix subsets $B_1, \dots, B_m \subseteq \{ 1, \dots, n \}$, such that for
all $1 \leqslant j \leqslant k \leqslant m$, $B_j \cap B_k$ is either
$B_j$ or empty. Set $X_B := \sum_{b \in B} X_b$ for all $B \subseteq \{
1, \dots, n \}$. If $S_n = X_1 + \cdots + X_n$, then for all $s,t>0$,
\begin{equation}
\bp{\max_{1 \leqslant k \leqslant m} d(0, 2 X_{B_k}) > s+t}
\leqslant \bp{d(0, S_n) > s} + \bp{d(0, S_n) > t}.
\end{equation}
\end{theorem}

Note that if $\G$ is a normed linear space and $s=t$, then the left-hand
side is concerned with the event that $\max_{1 \leqslant k \leqslant m}
\| 2 X_{B_k} \| > 2t$, which is how the inequality usually appears in the
literature.

It is the universal formulation and generalization of the result that is
of note here. Indeed, Theorem~\ref{Llevyineq} simultaneously strengthens
several different variants in the literature, which to our knowledge had
not previously been unified.
See \cite[Proposition 2.3]{LT} for two special cases where $\G$ is a
Banach space, $s=t$, $m=n$, and $B_k = \{ 1, \dots, k \}$ or $B_k = \{ k
\}$ for all $k$. Theorem~\ref{Llevyineq} also holds for other choices of
subsets
$B_k$, e.g.~$\{ 1 \}, \{ 1, 2 \}, \{ 3, 4, 5 \}, \{ 3, 4, 5, 6 \}$; or
$B_k = \{ n-k+1, \dots, n \}$ by reversing the order of summation; this
last choice is used below.
Moreover, Theorem~\ref{Llevyineq} extends L\'evy's inequality from Banach
spaces to all abelian metric groups. We provide a formal proof as it is
in a more general setting than what is available in the literature.

\begin{proof}[Proof of Theorem~\ref{Llevyineq}]
Define the stopping time
\begin{equation}
\tau := \min \{ 1 \leqslant k \leqslant m : d(0, 2 X_{B_k}) > s+t
\}.
\end{equation}
Now note that there is a smallest integer $1 \leqslant m_k \leqslant k$
such that $B_{m_k+1}, B_{m_k+2}, \dots, B_{k-1} \subseteq B_k$. By
assumption, $B_1, \dots, B_{m_k}$ are all disjoint from $B_k$.
Thus the event $\tau = k$, which denotes
\[
d(0, 2 X_{B_j}) \leqslant s+t < d(0, 2 X_{B_k}) \quad \forall
1 \leqslant j \leqslant k-1
\]

\noindent can be represented also as the event
\[
d(0, -2 X_{B_j}) \leqslant s+t \quad \forall 1 \leqslant j
\leqslant m_k, \qquad d(0, 2 X_{B_j}) \leqslant s+t < d(0, 2
X_{B_k}) \quad \forall m_k < j < k.
\]

\noindent Thus let ${\bf X}_r := X_{\cup_{j=1}^r B_j}$ for all $r$. Then
it follows from the symmetry assumption that
\begin{align*}
\bp{d(0, S_n) > t, \tau = k} = &\ \bp{d(0, (-{\bf X}_{m_k}) +
X_{B_k} - (S_n - {\bf X}_{m_k} - X_{B_k})) > t, \tau = k}\\
= &\ \bp{d(0, 2 X_{B_k} - S_n) > t, \tau = k}.
\end{align*}

We now prove the result. Note by the triangle inequality that
\[
d(0, 2 X_{B_k}) > s+t \quad \implies \quad d(0, S_n) > s
\quad \text{or} \quad d(0, 2 X_{B_k} - S_n) > t. \]

\noindent Thus by the above analysis, the result follows:
\begin{align*}
&\ \bp{\max_{1 \leqslant k \leqslant m} d(0, 2 X_{B_k}) > s+t} =
\sum_{k=1}^m \bp{\tau = k, \ d(0, 2 X_{B_k}) > s+t}\\
\leqslant &\ \sum_{k=1}^m \left( \bp{d(0, S_n) > s, \tau = k} +
\bp{d(0, 2 X_{B_k} - S_n) > t, \tau = k} \right)\\
= &\ \bp{d(0, S_n) > s,\ \tau \in [1,m]}
+ \bp{d(0, S_n) > t,\ \tau \in [1,m]}. \qedhere
\end{align*}
\end{proof}

We next show the Khinchin--Kahane inequality~\eqref{Ekhinchin}.

\begin{proof}[Proof of Theorem~\ref{Tkhinchin}]
For this proof, fix an abelian metric group $\G$, elements $x_1, \dots,
x_n$ $\in \G$, and Rademacher variables $r_1, \dots, r_n$. For ease of
exposition we break the proof into two steps.\medskip

\noindent \textbf{Step 1.}
We claim that \textit{for all abelian metric groups $\G$ and $\G$-valued
Rademacher sums $\sum_{k=1}^n r_k x_k$,}
\begin{align}\label{Elevy-appl}
\begin{aligned}
&\ \bp{d ( 0, 2 \sum_{k=1}^n r_k x_k ) > s+t+u+v}\\
& \leqslant (\bp{P_n > s} + \bp{P_n > t}) \cdot (\bp{P_n > u} + \bp{P_n >
v})
\end{aligned}
\end{align}

\noindent \textit{for all $s,t,u,v > 0$, where $P_n := d \left( 0,
\sum_{k=1}^n r_k x_k \right)$.}\medskip

Existing variants in the literature are usually special cases with $\G =
\bb$ a Banach space and $s=t=u=v$. While the proof uses familiar
arguments, we include it for the reader's convenience, as it is in
somewhat greater generality than can usually be found in the literature.

Define $S_k := \sum_{j=1}^k r_j x_j$ for $k \in [1,n]$.
Similar to the proof of L\'evy's inequality (Theorem~\ref{Llevyineq}),
define the stopping time
$\tau := \min \{ k \in [1,n] : d(0, 2 S_k) > s+t \}$.
Also recall that $(r_1, \dots, r_n)$ and $(r_1, \dots, r_k, r_k r_{k+1},
\dots, r_k r_n)$ are identically distributed. Therefore 
(using that $d(0,g) \equiv d(0,-g)$),
\begin{align*}
&\ \bp{d(2 S_{k-1}, 2 S_n) > u+v, \tau = k}
= \bp{d(0, 2 \sum_{j=k}^n r_j x_j) > u+v, \tau = k}\\
= &\ \bp{d(0, 2 \sum_{j=k}^n r_k r_j x_j) > u+v, \tau = k}\\
= &\ \bp{d(0, 2 x_k + 2 \sum_{j=k+1}^n r_k r_j x_j) > u+v, \tau = k}\\
= &\ \bp{d ( 0, 2 x_k + 2 \sum_{j=k+1}^n r_j x_j ) > u+v, \tau = k}
= \bp{d(2 x_k + 2 S_n, 2 S_k) > u+v, \tau = k}.
\end{align*}

\noindent The same argument without restricting to the event $\tau = k$
shows that:
\[
\bp{d(2 S_{k-1},2 S_n) > u+v} = \bp{d(2 x_k + 2 S_n, 2 S_k) > u+v}.
\]

\noindent Now note that if $d(0, 2 S_n(\omega)) > s+t+u+v$ and
$\tau(\omega) = k$, then $d(2 S_{k-1}(\omega), 2 S_n(\omega)) > u+v$.
Since $\tau = k$ and $d(2 S_k, 2 x_k + 2 S_n)$ are independent, we
compute:
\begin{align*}
&\ \bp{d(0, 2 S_n) > s+t+u+v, \tau = k} \leqslant
\bp{d(2 S_{k-1}, 2 S_n) > u+v, \tau = k}\\
= &\ \bp{d(2 x_k + 2 S_n, 2 S_k) > u+v} \bp{\tau = k}
= \bp{d(2 S_{k-1}, 2 S_n) > u+v} \bp{\tau = k}\\
\leqslant &\ \bp{\tau = k} \left( \bp{P_n > u} + \bp{P_n > v} \right),
\end{align*}

\noindent by using L\'evy's inequality (Theorem~\ref{Llevyineq}) with
$m=1, B_1 = \{ k, \dots, n \}, X_l = 2 r_l x_l\ \forall l \geqslant k$,
and replacing $(s,t)$ by $(u,v)$. Now another application of L\'evy's
inequality with the same choice of parameters -- except with $B_k = \{ 1,
\dots, k \}$ -- concludes the proof.
\comment{
{\color{red}
\begin{align*}
\bp{d(0, S_n^2) > s+t+u+v}
= &\ \sum_{k=1}^n \bp{d(0, S_n^2) > s+t+u+v, \tau = k}\\
\leqslant &\ \sum_{k=1}^n \bp{\tau = k} \left( \bp{P_n > u} + \bp{P_n >
v} \right)\\
= &\ \bp{\max_{1 \leqslant k \leqslant n} d(0, 2 S_k) > s+t} \left(
\bp{P_n > u} + \bp{P_n > v} \right)\\
\leqslant &\ \left( \bp{P_n > s} + \bp{P_n > t} \right) \cdot \left(
\bp{P_n > u} + \bp{P_n > v} \right).
\end{align*}
}}
\medskip

\noindent \textbf{Step 2.}
We now prove the inequality~\eqref{Ekhinchin} for $p,q \geqslant 1$.
Repeatedly applying the inequality~\eqref{Elevy-appl},
\begin{equation}
\bp{d(0, 2^l S_n) > 4^l r} \leqslant 4^{2^l-1} \bp{d(0, S_n)
> r}^{2^l}, \qquad \forall l \in \Z_{>0}.
\end{equation}

\noindent Set $l$ to be the unique positive integer such that $2^{l-1}
\leqslant q < 2^l$, and change variables $t = 4^l r \in (0,\infty)$.
Using that $\be_\mu[Z^q] = q \int_0^\infty t^{q - 1} \bp{Z>t}\ dt$ for an
$L^q$ random variable $Z \geqslant 0$, we compute:
\begin{align*}
\be_\mu[d(0, 2^l S_n)^q] = &\ q \int_0^\infty (4^l r)^{q-1}
\bp{d(0, 2^l S_n) > 4^l r} \cdot 4^l\ dr\\
\leqslant &\ q 4^{lq + 2^l-1} \int_0^\infty r^{q-1} \bp{d(0, S_n) >
r}^{2^l}\ dr.
\end{align*}

\noindent Now $4^{lq} \leqslant (2q)^{2q}$ and $r \bp{d(0,S_n) > r}
\leqslant \be_\mu[d(0,S_n)]$ by Markov's inequality. Therefore,
\begin{align*}
\be_\mu[d(0, 2^l S_n)^q] \leqslant &\ (2q)^{2q} 4^{2q-1} q
\int_0^\infty \be_\mu[d(0,S_n)]^{q-1} \cdot \bp{d(0,S_n) >
r}\ dr
= \frac{(8q)^{2q+1}}{32} \be_\mu[d(0, S_n)]^q.
\end{align*}

\noindent Taking $q$th roots and using H\"older's inequality now
yields~\eqref{Ekhinchin}.
\end{proof}
%}}}

\section{The transfer principle for normed (abelian semi)groups and its
applications}\label{Snormed}

In this section we formulate and prove the transfer principle promised
above, which will allow one to take random variables and their
probability/functional analysis from a subclass of abelian metric
semigroups to Banach spaces. We begin by introducing the key notion
required for this.

\begin{definition}\label{Dnormed}
We say that a (possibly non-abelian) metric semigroup $(\G, \cdot, d)$ is
\textit{normed} if
\begin{equation}\label{Enorm}
d(g, g^{n+1}) = n d(g,g^2), \qquad \forall g \in \G, \ n \in
\Z_{\geqslant 0}.
\end{equation}
\end{definition}

Notice that if $\G$ is an abelian metric group, then~\eqref{Enorm}
implies the following stronger version:
\[
d(ng, nh) = |n| d(g,h), \qquad \forall g,h \in \G, \ n \in \Z.
\]

There is extensive literature on the analysis of topological semigroups
with translation-invariant metrics and related structures. See
\cite{BO,CSC,GK} and the references therein for more on the subject.
These references call any group with a metric (under which the inverse
map is an isometry) a ``normed'' group, while the above
condition~\eqref{Enorm} is termed \textit{$\Z_{>0}$-homogeneity}.
However, in Definition~\ref{Dnormed} we instead adopt the notation of
\cite{Step}, and define a norm to be more in the flavor of Banach spaces.
The objects in Definition~\ref{Dsemi} will be called metric (semi)groups
in this paper.

%{{{1 Section 3.1 - The transfer principle and its applications to
%abelian normed  semigroups
\subsection{The transfer principle and its applications to abelian normed
semigroups}\label{Sappl}

We now return to our extension~\eqref{Ekhinchin} of the traditional
Khinchin--Kahane inequality. Notice that if the abelian metric group $\G$
is moreover normed, then the left-hand side of~\eqref{Ekhinchin} equals a
scalar times the original left-hand side in~\eqref{Ekhinchin-Banach}, and
so it is natural to ask if Theorem~\ref{Tkhinchin} can be modified to
yield the same (improved) constants as in the Banach space
case~\eqref{Ekhinchin-Banach}.

It turns out that this is indeed possible, as we explain presently using
our next result, Theorem~\ref{Tnormed}, which says that $\G$ ``is'' a
subspace/additive subgroup of a Banach space -- and constructs a
candidate for the ``smallest'' such Banach space. We show this result for
all abelian normed \textit{semigroups} $\G$. Of course, the
Khinchin--Kahane inequality works with metric groups, so only the special
case of our next result, wherein $\G$ is a(n abelian) normed group, is
required to address our motivation in the preceding paragraph. This
special case can be found in the literature, see below.

\begin{theorem}[Transfer principle]\label{Tnormed}
Every (separable) abelian normed metric semigroup $\G$ canonically and
isometrically embeds into a ``smallest'' (separable) real Banach space
$\bb(\G)$. The same holds if $\G$ is a normed (but not \textit{a priori}
abelian) metric group. In particular, the theory of Bochner integration
and expectations extends to all such (semi)groups $\G$: if $X$ is
$\G$-valued then $\be_\mu[X] \in \bb(\G)$.
\end{theorem}

\begin{remark}
While the results in this work hold only for groups that are abelian, we
stress that the abelian hypothesis is \textit{not} required in the second
assertion in Theorem~\ref{Tnormed}. See Theorem~\ref{Tpolymath}.
\end{remark}

The proof of Theorem~\ref{Tnormed} is related to previous constructions
in the literature, and we defer it to later in this section. At present,
we discuss some of its consequences and applications.

\begin{example}
The first application is that our Khinchin--Kahane
inequality~\eqref{Ekhinchin} can be refined for \textit{normed} metric
groups $\G$ to yield precisely the (sharp) constants in the Banach space
setting:

\begin{prop}\label{Pkhinchin}
For all normed groups $\G$, one has the ``usual'' Khinchin--Kahane
inequality, with a universal constant $C_{p,q}$ (for fixed $p,q \geqslant
1$ but universal across all normed $\G$ and $n, x_k, r_k$):
\begin{equation}\label{Ekhinchin-normed}
\be_\mu \left[ d \left( 0, \sum_{k=1}^n r_k x_k \right)^q
\right]^{1/q} \leqslant C_{p,q} \cdot \be_\mu \left[ d \left( 0,
\sum_{k=1}^n r_k x_k \right)^p \right]^{1/p}.
\end{equation}
Moreover, for all $p,q$, the constant $C_{p,q}$ (universal over the
category of all abelian normed groups) is equal to the universal constant
when working only with the sub-category of all real Banach spaces.
\end{prop}

Recall that in the classic paper~\cite{LO}, Lata{\l}a--Oleszkiewicz
obtained the optimal such universal constant across all Banach spaces in
the regime $p=1 \leqslant q \leqslant 2$, namely, $C_{1,q} = 2^{1 -
1/q}$. Proposition~\ref{Pkhinchin} shows that $C_{1,q}$ also works for
the Khinchin--Kahane inequality in (abelian) normed metric groups. The
proof is immediate: consider all $\G$-valued random variables to now be
$\bb(\G)$-valued.
\end{example}

We now explain several other applications -- all of which involve abelian
normed \textit{semigroups} (noting that such results/applications are not
discussed even for normed groups in the literature):

\begin{example}
More broadly than Proposition~\ref{Pkhinchin}, Theorem~\ref{Tnormed}
provides a route to ``transfer'' problems from abelian normed metric
semigroups to Banach spaces. For instance, L\'evy's equivalences between
modes of stochastic convergence of sums of independent $\G$-valued random
variables immediately follow from their Banach space counterparts for
$\bb(\G)$, e.g.~\cite[Theorem 2.4]{LT}.
\end{example}

\begin{example}
A third -- and more challenging -- application of Theorem~\ref{Tnormed}
is to extend to normed $\G$ the main result of \cite{DM}, which provides
universal constants that occur in bounding vector-valued Rademacher
series. We now extend this theorem to arbitrary normed $\G$ (and the
$K^w_{1,2}$ in the statement of the next result will be explained
following Corollary \ref{Cdm}). Note that such an extension result is not
immediate as one has to first understand better the notion of ``linear
functionals'' on $\G$. This is carried out below; in what follows, $\| g
\|$ denotes $d(0,g)$.

\begin{theorem}\label{Tdm}
Fix an i.i.d.\ sequence of Rademacher variables $\varepsilon_n \sim {\rm
Unif} \{ -1,1 \}$. Then there exists an absolute constant $c>0$ such that
for all choices of 
(a)~separable abelian normed metric semigroups $\G$,
(b)~points $g_n \in \G$ such that the Rademacher series $X := \sum_n
\varepsilon_n g_n$ is almost surely convergent (e.g.\ in $\bb(\G)$), and
(c)~scalars $t>0$, we have:
\begin{align}
\bp{\| X \| > 2 \be \| X \| + 6 K^w_{1,2}((g_n), t)} \leqslant &\ 4
e^{-t^2/8},\\
\bp{\| X \| > \frac{1}{2} \be \| X \| + c K^w_{1,2}((g_n), t)} \geqslant
&\ c e^{-t^2/c}.
\end{align}
\end{theorem}

Observe that the results of Talagrand \cite{Ta} that are cited in
\cite{DM} also extend to abelian normed semigroups, as does the
observation that opens the proof of the main theorem in \cite{DM}:

\begin{prop}
All separable abelian normed semigroups are
``isometric sub-semigroups'' of $\ell_\infty$.
\end{prop}

Furthermore, the various applications of (the version of)
Theorem~\ref{Tdm} in \cite{DM} also hold in abelian normed semigroups.
These include the following ``semigroup-valued'' precise form of the
Khinchin--Kahane inequality, in a sense bringing us back full circle to
Theorem~\ref{Tkhinchin}.

\begin{cor}[{see \cite[Corollary~3]{DM}}]\label{Cdm}
As above, let $X := \sum_n \varepsilon_n g_n$ be an almost surely
convergent Rademacher series with all $g_n$ in a separable abelian normed
metric semigroup $\G$. Then there is an absolute constant $c > 0$ such
that
\[
\frac{1}{c} \be[ \| X \|^p ]^{1/p} \leqslant
\be \| X \| + K^w_{1,2}((g_n), \sqrt{p}) \leqslant
c \be[ \| X \|^p ]^{1/p}, \quad \forall p \in [1,\infty).
\]
\end{cor}
\end{example}

Let us explain the preceding theorem (and hence its corollary), and in
particular, why these results are not direct applications of the transfer
principle in Theorem~\ref{Tnormed}.
In Theorem~\ref{Tdm} and Corollary~\ref{Cdm}, the constant
$K^w_{1,2}((g_n),t)$ was used for scalars $t>0$. In the original setting
of \cite{DM}, defining this constant involves Banach space analysis and
weakly $\ell_p$ sequences. We now extend this definition to all abelian
normed semigroups $\G$. For $p \in [1,\infty)$, we say a sequence of
points $(g_n)_n$ in $\G$ is \textit{weakly} $\ell_p$ if $(g^*(g_n))_n$ is
$\ell_p$ for every $g^* \in \G^*$, where $\G^*$ denotes the set of
additive Lipschitz real-valued maps on $\G$. Note, this differs from the
Banach space definition, which would require running over all functionals
in $\bb(\G)^*$ (or the dual space to a larger Banach space), via
Theorem~\ref{Tnormed}.

Now define for a weakly $\ell_2$ sequence $(g_n)$ and a scalar sequence
$(a_n) \in \ell_2$:
\begin{align*}
K_{1,2}((a_n),t) := &\ \inf \{ \| (a_{1,n}) \|_1 + t \| (a_{2,n}) \|_2 :
a_n = a_{1,n} + a_{2,n}\ \forall n, \ (a_{j,n})_n \in \ell_j \text{ for }
j = 1,2 \},\\
K^w_{1,2}((g_n),t) := &\ \sup \{ K_{1,2}((g^*(g_n))_n, t) : g^* \in \G^*,
\| g^* \| \leqslant 1 \}.
\end{align*}

\noindent Then the key is that the $K^w_{1,2}$-value computed using
$\G^*$ exactly matches the Banach space version that uses $\bb(\G)^*$
(hence the results of \cite{DM} extend to abelian normed semigroups), by
our next result:

\begin{prop}\label{Psame-dual}
For $\G$ an abelian normed semigroup, let $\G^*$ denote the set of
additive Lipschitz real-valued maps on $\G$. Then $\G^*$ is a Banach
space, which coincides with the dual space construction if $\G$ is a
Banach space. More generally if $\G$ is an abelian normed semigroup, then
$\G^* \simeq \bb(\G)^*$.
\end{prop}

As this paper focuses on probability inequalities and analysis-related
constructions, we defer the proof of Proposition~\ref{Psame-dual} to a
standalone appendix on a \textit{category-theoretic} treatment of normed
\textit{modules} and their properties, for the interested reader; see
Proposition~\ref{Pdual}. In particular, as noted in \cite{DM}, the
assignment $t \mapsto K^w_{1,2}((g_n),t)$ is Lipschitz with Lipschitz
constant at most
\[
\ell^w_2((g_n)) := \sup_{\| g^* \| \leqslant 1} \| (g^*(g_n)) \|_2
\]

\noindent (where $g^*$ runs over $\G^*$), and Theorem~\ref{Tdm} holds
over all abelian normed metric semigroups.

\begin{example}
As additional consequences of our ``transfer principle'' in
Theorem~\ref{Tnormed}, the main results in \cite{Mo,MP} immediately
extend to arbitrary abelian normed semigroups.
\end{example}
%}}}

%{{{1 Section 3.2 - Banach space embeddings
\subsection{Banach space embeddings}

We next return to (the proof of) Theorem~\ref{Tnormed}, and discuss it
vis-a-vis the question of embedding a given topological group into a
Banach space. The theorem says that for a metric (semi)group $(\G, \cdot,
d)$, the assumption of being (abelian and) normed is sufficient to embed
$\G$ into a Banach space. Clearly, the assumption is also necessary. The
next result provides additional equivalent conditions when $\G$ is a
group, and also relates it to results in the literature.

\begin{definition}
Given $J \subseteq \Z_{>0}$, a (possibly non-abelian) metric semigroup
$(\G, d)$ is {\em $J$-normed} if
\begin{equation}
d(z_0, z_0^{n+1}) = n d(z_0, z_0^2),
\qquad \forall z_0 \in \G, n \in J.
\end{equation}
\end{definition}

\begin{prop}\label{Pnormed}
Suppose $\G$ is a topological group, with a continuous map $\| \cdot \| :
\G \to [0,\infty)$ satisfying:
(a)~$\| g \| = 0$ if and only if $g = 1$;
(b)~$\| g^{-1} \| = \| g \|$ for all $g \in \G$; and
(c)~the triangle inequality holds: $\| gh \| \leqslant \| g \| + \| h \|$
for $g,h \in \G$.
Then the following are equivalent:
\begin{enumerate}
\item There exists a Banach space $\bb$ and a group map $: \G \to (\bb,
+)$ that is an isometric embedding.

\item $\G$ is abelian and $d(g,h) := \| g^{-1} h \|$ is a
translation-invariant metric for which $\G$ is normed.

\item $\G$ is $\{ 2 \}$-normed and is {\em weakly commutative}, i.e., for
all $g,h \in \G$ there exists $n = n(g,h) \in \Z_{>0}$ such that
$(gh)^{2^n} = g^{2^n} h^{2^n}$.

\item $\G$ is $\{ 2 \}$-normed and amenable.
\end{enumerate}
\end{prop}

In fact there is a fifth (\textit{a priori} weaker than~(2) or~(4), yet)
equivalent condition -- that $\G$ is $\{ 2 \}$-normed \textit{without}
additional restrictions -- which we explain in Theorem~\ref{Tpolymath}.

\begin{proof}
That $(1) \implies (2) \implies (3)$ is immediate. That~$(3)$ or~$(4)$
implies $(1)$ follows from \cite[Proposition~4.12]{CSC} via
\cite[Corollary~1]{GK}. This is a constructive proof, and the formula for
the Banach space in question is discussed later in this subsection.
Finally, that $(1) \implies (4)$ follows since every abelian group is
amenable (see \cite{Da,Pat,Wag} for more on amenable groups).
\end{proof}

In this connection, the following result shows (as a special case) that
even without requiring the semigroup to be abelian, the ``normed''
property of a translation-invariant metric on a semigroup:
\[
d(z_0, z_0^{n+1}) = n\, d(z_0, z_0^2), \qquad \forall z_0 \in \G, \
n \in \Z_{\geqslant 0}
\]
already follows from -- hence is equivalent to -- the ``doubling''
property of being $\{ 2 \}$-normed:
$d(z_0, z_0^3) = 2 d(z_0, z_0^2)$ for all $z_0 \in \G$.
We omit the proof as it is a variant of \cite[Lemma~1]{GK}.

\begin{lemma}\label{Lnormed}
Given a nonempty subset $J \subseteq \Z_{>0}$, $J \neq \{ 1 \}$, a metric
semigroup $\G$ is $J$-normed if and only if $\G$ is $\Z_{>0}$-normed.
\end{lemma}

We next constructively prove the ``transfer principle'' above.

\begin{proof}[Proof of Theorem~\ref{Tnormed}]
The first point is that every normed metric group is abelian by
Theorem~\ref{Tpolymath} below (see \cite{Po}), and so the second
assertion reduces to the first. To show the first assertion, we will use
additive notation throughout this proof as $\G$ is abelian. The proof is
constructive, and carried out in stages; however, an outline is in the
following equation:
\begin{equation}\label{Enormed}
\G_\N := \G \hookrightarrow \G_{\N \cup \{ 0 \}} := \G' \hookrightarrow
\G_\Z := \Z \otimes_{\N \cup \{ 0 \}} \G_{\N \cup \{ 0 \}}
\hookrightarrow \G_\Q := \Q \otimes_\Z \G_\Z \hookrightarrow \bb(\G) :=
\overline{\G_\Q},
\end{equation}

\noindent where $\N := \Z_{>0}$. We now explain these steps one by one.
\begin{enumerate}
\item Embed the semigroup into a metric monoid $\G'$ via
Corollary~\ref{Cstrict}. We label $\G_\N := \G$ and $\G_{\N \cup \{ 0 \}}
:= \G'$ to denote that $\G,\G'$ are ``modules'' over $\N, \N \cup \{ 0
\}$ respectively.

\item It is easily shown that $\G_\N$ and hence $\G_{\N \cup \{ 0 \}}$ is
cancellative. 
Therefore the monoid $\G_{\N \cup \{ 0 \}}$ embeds into its Grothendieck
group\footnote{The \textit{Grothendieck group} of an abelian monoid $M$
is the abelian group $Gr_M$ of equivalence classes in $M \times M$ under:
$(a,a') \sim (b,b')$ if there exists $m \in M$ such that $a+b'+m =
a'+b+m$. The element/class $[(m,m')]$ should be thought of as $m-m'$, the
operation is coordinatewise addition, the zero is $[(m,m)]$ for any $m
\in M$, and the inverse of $[(m,m')]$ is $[(m',m)]$. If $M$ has the
cancellation property then $M \hookrightarrow Gr_M$.}
$\G_\Z$ (which is a $\Z$-module) by attaching additive inverses and
quotienting by an equivalence relation. Extend the metric $d_{\G_{\N \cup
\{ 0 \}} }$ to $\G_\Z$ via: $d_{\G_\Z}(p-q, r-s) := d_{\G_{\N \cup \{ 0
\}} }(p+s, q+r)$, for all $p,q,r,s \in \G_{\N \cup \{ 0 \}}$.
Then $(\G_\Z, 0_{\G_{\N \cup \{ 0 \}} }, d_{\G_\Z})$ is an abelian metric
group and $\G_\N \hookrightarrow \G_{\N \cup \{ 0 \}} \hookrightarrow
\G_\Z$ are isometric (hence injective) semigroup/monoid homomorphisms.
$\G_\Z$ is also normed since for all $n \in \Z$ and all $p,q \in \G_{\N
\cup \{ 0 \}}$,
\[
d_{\G_\Z}(0_{\G_{\N \cup \{ 0 \}} }, n (p-q)) = d_{\G_{\N \cup \{ 0
\}} }(|n|q, |n|p) = |n| d_{\G_{\N \cup \{ 0 \}} }(q,p) = |n|
d_{\G_\Z}(0_{\G_{\N \cup \{ 0 \}} }, p-q).
\]

\item Note that $\G_\Z$ is a torsion-free $\Z$-module -- i.e., if $ng =
0_{\G_\Z}$ for $n \in \Z \setminus \{ 0 \}$ and $g \in \G_\Z$, then $g =
0_{\G_\Z}$. Now define $\G_\Q := \Q \otimes_\Z \G_\Z$; thus $\G_\Q$ is a
$\Q$-vector space (i.e., a torsion-free divisible\footnote{These terms --
as also tensor products -- are defined and studied in standard algebra
textbooks, e.g.\ \cite{Lang}.} group), and $\G_\Z$ embeds into $\G_\Q$.
Moreover for every $g \in \G_\Q$ there exists $n_g \in \N$ such that $n_g
g \in \G_\Z$. Define $d_{\G_\Q}$ on $\G_\Q^2$ via:
\[
d_{\G_\Q}(g,h) := \frac{1}{n_g n_h} d_{\G_\Z}(n_h (n_g g),
n_g (n_h h)).
\]

\noindent It is not hard to check that $d_{\G_\Q}$ is well-defined and
induces a ``$\Q$-norm'' on $\G_\Q$ that extends $d_{\G_\Z}$ on $\G_\Z$.
In particular, it induces a translation-invariant metric on $\G_\Q$, so
that we have embedded the normed semigroup $\G_\N$ isometrically into a
normed $\Q$-vector space.

\item Define $\bb(\G)$ to be the set of equivalence classes of
$d_{\G_\Q}$-Cauchy sequences (i.e., the topological completion) of
$\G_\Q$. One shows using algebraic and topological arguments that
$\bb(\G)$ is an abelian group and $\G_\Q$ embeds isometrically into
$\bb(\G)$. Moreover, if $x \in \R$ and $(g_n)_n$ is Cauchy in $\bb(\G)$,
then choose any sequence $x_n \in \Q$ converging to $x$, and define $x
\cdot [(g_n)_n] := [(x_n g_n)_n]$. It is easy to verify that $(x_n
g_n)_n$ is also a Cauchy sequence in $\G_\Q$, and the resulting operation
makes $\bb(\G)$ into an $\R$-vector space.

Now define $d_{\bb(\G)}([(g_n)_n], [(h_n)_n]) := \lim_{n \to \infty}
d_{\G_\Q}(g_n,h_n)$ (this exists and is well-defined by applying
topological arguments). It is easily verified that $d_{\bb(\G)}$ induces
a norm on $\bb(\G)$, making $\bb(\G)$ a complete normed linear space, and
proving~\eqref{Enormed}.
\end{enumerate}

To conclude the proof, observe that if any of the steps starts with a
separable metric space, then the subsequent constructions also yield
separable metric spaces. 
The final assertion about extending Bochner integration to $\G$ now
follows; note the Bochner integral (or expectation) of $\G$-valued random
variables lives in $\bb(\G)$ and not necessarily in $\G$.
\end{proof}

\begin{remark}
If one restricts Theorem~\ref{Tnormed} to groups instead of semigroups,
Proposition~\ref{Pnormed} from~\cite{CSC,GK} is related, as it embeds an
abelian normed group into \textit{some} Banach space, via the ``double
dual''. To our knowledge, an explicit construction of a \textit{minimal}
Banach space ``envelope'' was not recorded to date in the literature. We
briefly discuss this and other aspects of our proof:
\begin{enumerate}
\item There was no mention of separability in the proof in~\cite{CSC}.
This is useful for applications (see Section~\ref{Sappl}) and hence is
addressed by our proof.

\item The construction in~\cite[Proposition~4.12]{CSC} is that of the
``double-dual''
$\bb := {\rm Hom}_{gp,bdd}(\G, \R)^*$,
i.e., the dual space to the set of bounded/Lipschitz real-valued group
maps $: \G \to \R$. Thus, if $\G$ is an infinite-dimensional Banach
space, then the double-dual construction $\G^{**}$ is strictly larger
than $\G$ -- and hence, does not yield the ``minimal'' Banach space
envelope of $\G$ for ``most'' real Banach spaces -- whereas the above
proof does. One of the referees mentioned to us that the construction can
be refined to yield the minimal Banach space; however, to our knowledge
this refinement is not written down. This was one reason to write the
above argument in full detail -- especially given that our constructive
proof is along \textit{different} lines.

\item To the best of our knowledge, we could not find references to
embedding abelian normed \textit{semigroups}. For this a little more work
is needed to embed into an abelian normed group (via the monoid-extension
and then its Grothendieck group, as above). This semigroup-extension also
features in several applications (in Section~\ref{Sappl}), hence its
proof above.
\end{enumerate}
\end{remark}

We end with two further remarks. First, each step in~\eqref{Enormed} is
canonical (and optimal), in the sense that it uses only the given
information without any additional structure (and it constructs the
``minimal'' larger structure containing the structure at each step).
The natural way to encode this information is via category theory. In
other words, every further step/extension in~\eqref{Enormed} is the
smallest possible -- hence \textit{universal} -- ``enveloping'' object in
some category. For the interested reader, we defer these categorical
discussions to Appendix~\ref{Acat}.

Second, given Corollary~\ref{Cstrict}, it is natural to ask in the
non-abelian situation if every (cancellative) metric semigroup embeds
into a metric group. This question is harder to tackle; see \cite[Chapter
1]{CP} for a sufficient condition involving right reversibility.
%}}}

%{{{1 Section 3.3 - Non-abelian normed groups
\subsection{Non-abelian normed groups}\label{Snonabelian}

We end this section with a geometric question that clarifies the second
assertion in Theorem~\ref{Tnormed}:
\textit{Do non-commutative normed metric groups exist?}
Or: find a non-abelian topological group $\G$ with a bi-invariant metric
$d$, such that $d(1, g^n) = |n| d(1,g)$ for all $g \in \G$ and $n \in
\Z$. To our knowledge (and that of experts including \cite{GaKe,St2} and
others), the answer to this question was not known until recent work
\cite{Po}, whose main result we now describe.

As a possible approach to answering the aforementioned question, a first
step is to ask if certain prototypical examples of non-commutative groups
with a bi-invariant metric are normed. This is now shown to be false for
a well-studied example: the \textit{free group} $F_2$ on two generators.
Recall, this is simply the set of words in $a, b, a^{-1}, b^{-1}$, modulo
the relations $a a^{-1} = a^{-1} a = b b^{-1} = b^{-1} b = 1$.

\begin{lemma}\label{Lword}
Let $\G = F_2$ and let $d_\G$ denote the bi-invariant word metric in the
generators $a^{\pm 1}, b^{\pm 1}$ and their conjugates. Then $(\G, \cdot,
1, d_\G)$ is not normed.
\end{lemma}

Note that we work with $d_\G$ and not the usual word metric in the
four semigroup generators $a^{\pm 1}, b^{\pm 1}$ of $\G$. For more on the
metric $d_\G$ and related structures, see~\cite{BGKM} and the references
therein.

\begin{proof}
First compute: $[a,b]^3 = a b a^{-1} \cdot b^{-1} a b \cdot a^{-1} b^{-1}
a \cdot b a^{-1} b^{-1}$.
Examining the word lengths, the right-hand side yields at most $4$, while
$l_\G([a,b]) \neq 1$. Hence $(\G, \cdot, 1, d_\G)$ is not normed, as
\[
l_\G([a,b]^3) \leqslant 4 < 6 \leqslant 3 l_\G([a,b]). \qedhere
\]
\end{proof}

We conclude with a solution to the above question, obtained by the first
author in recent joint work \cite{Po} with T.~Fritz, S.~Gadgil,
P.~Nielsen, L.~Silberman, and T.~Tao. It turns out that non-commutative
normed metric groups do not exist! Namely:

\begin{theorem}[\cite{Po}]\label{Tpolymath}
Given a group $\G$, the following are equivalent:
\begin{enumerate}
\item $\G$ is a metric group (with a bi-invariant metric) that is $\{ 2
\}$-normed (equivalently, normed).

\item $\G$ is abelian and torsion-free.

\item $\G$ is an additive subgroup of (i.e., embeds isometrically and
additively into) a Banach space.
\end{enumerate}
\end{theorem}

This yields a novel characterization -- from analysis -- of a fundamental
class of algebraic objects: abelian torsion-free groups. 
Now given Theorem~\ref{Tpolymath}, many of the above results (e.g.\ our
transfer principle in Theorem~\ref{Tnormed}) are stated for normed metric
groups $\G$, as a norm on $\G$ implies $\G$ is abelian. Note however that
this last can fail if ``group'' is replaced by ``semigroup'' or even
``monoid'', since non-abelian (free) monoids with norms -- i.e.,
homogeneous length functions -- indeed exist; see \cite{Po} for details.
Another consequence is that the four assertions in
Proposition~\ref{Pnormed} are further equivalent to an \textit{a priori}
weaker assertion than~(2) or~(4): namely, that $\G$ is merely $\{ 2
\}$-normed.
%}}}

\appendix

%{{{1 Appendix A - Categories of normed metric modules
\section{Categories of normed metric modules}\label{Acat}

We now construct ``dual spaces'' to abelian normed metric groups, in
order to prove Proposition~\ref{Psame-dual} (see the discussion after
Theorem~\ref{Tdm}). This construction is similar -- with minor
adjustments -- for abelian normed metric semigroups and their
refinements:
(i)~semigroups,
(ii)~monoids,
(iii)~groups,
(iv)~torsion-free divisible groups,
(v)~real vector spaces, and
(vi)~Banach spaces.
To study all of these constructions systematically, we use the
language of category theory. (For basics of categories and functors, see
\cite{Lang}.) We will show in Proposition \ref{Pdual} below that ``dual
space constructions'' are covariant endofunctors -- more generally,
so are spaces of linear Lipschitz operators.

Using categories has additional advantages.
Recall that the proof of Theorem~\ref{Tnormed} showed that every abelian
normed semigroup (respectively, group) embeds into a smallest abelian
normed group (respectively, Banach space). We now make these statements
precise using category theory. Briefly, we will show in a unified way
that the above constructions are instances of ``universal objects'', and
provide examples of pairs of adjoint ``induction-restriction functors''.

To proceed, we first propose a unifying framework in which to
simultaneously study abelian normed metric semigroups of types (i)--(vi)
above: normed metric modules. In the sequel, $\N = \Z_{>0}$.

\begin{definition}
Suppose a subset $S \subseteq \R$ is closed under addition and
multiplication.
\begin{enumerate}
\item An \textit{$S$-module} is defined to be an abelian semigroup
$(G,+)$ together with an action map $\cdot : S \times G \to G$,
satisfying the following properties for $s,s' \in S$ and $g,g' \in
G$:\footnote{Note that if $0 \in S$ then $G$ is necessarily a monoid.}
\[
s \cdot (g+g') = s \cdot g + s \cdot g', \quad (s+s') \cdot g = (s \cdot
g) + (s' \cdot g), \quad (s s') \cdot g = s \cdot (s' \cdot g), \quad 1
\cdot g = g \text{ if } 1 \in S.
\]

\item A \textit{metric $S$-module} is an $S$-module $(G,+)$ together with
a translation-invariant metric $d$. We say $(G,+,d)$ is \textit{normed}
if $d(s \cdot g, s \cdot g') = |s| d(g,g')$ for all $s \in S$ and $g,g'
\in G$.

\item Let $\scrc_S$ denote the category whose objects are normed metric
$S$-modules $G_S$, and morphisms are $S$-module maps that are moreover
Lipschitz. For each such morphism $\varphi : G_S \to G'_S$, define $\|
\varphi \|$ to be the smallest constant $K \geqslant 0$ such that $\|
\varphi(g) \| \leqslant K \| g \|$ for all $g \in G_S$.
Also denote by $\overline{\scrc}_S$ the full sub-category of all objects
in $\scrc_S$ that are complete metric spaces.
\end{enumerate}
\end{definition}

Now $\N$-modules are semigroups and $(\N \cup \{ 0 \})$-modules are
monoids. Using this notation, Theorem~\ref{Tnormed} discusses the objects
in the categories $\scrc_S$ for $S = \N, \N \cup \{ 0 \}, \Z, \Q$, as
well as $\overline{\scrc}_\R$, the category of Banach spaces and bounded
operators. Note that normed linear spaces (i.e., $\scrc_\R$) are missing
from Theorem~\ref{Tnormed}; however, our next result also produces a
similar ``universal'' normed linear space containing a(n abelian) normed
group. Thus, the constructions in~\eqref{Enormed} possess functorial
properties and therefore are universal in the above categories:

\begin{theorem}\label{Tuniversal}
Suppose each of $S,T,U$ is either $\N, \N \cup \{ 0 \}$, or a unital
subring of $\R$, with $S \subseteq T$ or $S \supseteq T$. Suppose also
that $G_S$ is an object of $\scrc_S$. Now define
\begin{equation}
\G_T(G_S) := \begin{cases}
    G_S \ (\text{viewed as an object of } \scrc_T), & \text{if } S
    \supseteq T;\\
    \text{the unique object of $\scrc_T$ defined as in~\eqref{Enormed}},
    & \text{if } S = \N, \N \cup \{ 0 \},
    \ T \supseteq S;\\
    T \otimes_S G_S, & \text{if } \Z \subseteq S \subseteq T.
\end{cases}
\end{equation}
\begin{enumerate}
\item $\G_T(G_S)$ is an object of $\scrc_S \cap \scrc_T$ with the
following universal property: given an object $G_T$ in $\scrc_S \cap
\scrc_T$, together with a morphism $\iota : G_S \to G_T$ in $\scrc_S$,
$\iota$ extends via the unique isometric monomorphism $G_S
\hookrightarrow \G_T(G_S)$ to a unique morphism $\iota_T : \G_T(G_S) \to
G_T$ in $\scrc_T$.

\item In particular, $(\G_T(G_S), \iota_T)$ is unique up to a unique
isomorphism in $\scrc_T$.

\item Given $G_S$, define $\overline{\G}_T(G_S)$ to be the Cauchy
completion of $\G_T(G_S)$ (as a metric space). Then
$\overline{\G}_T(G_S)$ is an object of $\overline{\scrc}_T$ and satisfies
the same properties as in the previous parts.

\item Suppose $\N \subseteq S \subseteq T \subseteq U \subseteq \R$, with
$S,T,U$ of the form $\N, \N \cup \{ 0 \}$, or a unital subring of $\R$.
For all objects $G_S$ in $\G_S$, there exist unique isomorphisms:
\[
\G_U(\G_T(G_S)) \cong \G_U(G_S), \qquad \overline{\G}_U(\G_T(G_S)) \cong
\overline{\G}_U(\overline{\G}_T(G_S)) \cong \overline{\G}_U(G_S).
\]

\item Given a unital subring $S \subseteq \R$,
$S$ is dense in $\R$ if and only if
$\overline{G_S} = \overline{\G}_T(G_S) = \bb(G_S)$ for all objects
$G_S$ of $\scrc_S$ and all subrings $S \subseteq T \subseteq \R$.
(Here, $\bb(G_S)$ is as in Theorem~\ref{Tnormed}.)
\end{enumerate}
\end{theorem}

For the above reason, if $S \subseteq T$ or $S \supseteq T$ then we call
$\G_T(G_S), \overline{\G}_T(G_S)$ the {\em universal envelopes} of $G_S$
in $\scrc_T$ and $\overline{\scrc}_T$ respectively.\footnote{Such
``minimal envelopes'' are ubiquitous in mathematics; examples include the
universal enveloping algebra of a Lie algebra, the convex hull of a set
(in a real vector space), and the $\sigma$-algebra generated by a set of
subsets.} Also observe that $\bb(G_S)$ is the completion of the smallest
normed linear space containing $G_S$, for all $S \supseteq \mathbb{Q}$
and objects $G_S$ in $\scrc_S$.

\begin{proof}[Proof of Theorem~\ref{Tuniversal}]
The proof involves (sometimes standard) category-theoretic arguments, and
is included for the convenience of the reader.
\begin{enumerate}
\item The first part is immediate if $S \supseteq T$; we now show it
assuming that $S \subseteq T$. Given an object $G_S$ in $\scrc_S$, note
$\G_T(G_S) \subseteq \bb(G_S)$. This immediately shows $\G_T(G_S)$ is an
object of $\scrc_T$. Now given a morphism $\iota : G_S \to G_T$ in
$\scrc_S$, if $S = \N$ then first define $\iota_T(0_{\G_T(G_S)}) :=
0_{G_T}$. If $S = \N \cup \{ 0 \}$ then define $\iota_T(-g) := -\iota(g)$
for $g \in G_S$. Finally, if $S$ is a unital subring of $\R$ and $x :=
\sum_{j=1}^n t_j g_j \in T \otimes_S G_S$ (with $g_j \in G_S\ \forall
i$), then define $\iota_T(x) := \sum_{j=1}^n t_j \iota(g_j)$. These
conditions are necessary to extend $\iota$ to $\iota_T$; moreover, it is
not hard to show using Theorem~\ref{Tnormed} that they are also
sufficient to uniquely extend $\iota$ to $\iota_T$. Also using
Theorem~\ref{Tnormed}, one verifies that $\iota_T$ is Lipschitz, with $\|
\iota_T \| = \| \iota \|$.

\item This is a standard categorical consequence of universality.

\item This is clear if $S \supseteq T$, so say $S \subseteq T,\ G_S \in
\scrc_S$. Given $\iota : G_S \to G_T$ with $G_T \in \scrc_S \cap
\overline{\scrc}_T$, by (1) $\iota$ extends uniquely to $\iota_T :
\G_T(G_S) \to G_T$, which in turn extends uniquely to $\overline{\iota_T}
: \overline{\G}_T(G_S) \to G_T$ by uniform continuity. Now verify
$\overline{\iota_T}$ is a morphism in $\overline{\scrc}_T$, with $\|
\overline{\iota_T} \| = \| \iota_T \|$.

\item This part is standard from above using universal properties, and is
omitted for brevity.

\item First if $S$ is not dense in $\R$, i.e.~$S = \Z$, then choose $G_S
= \Z$. Now $\overline{G_S} = \Z \neq \R = \bb(G_S)$, as asserted.
Conversely, suppose $S$ is dense in $\R$ and $G_S$ is in $\scrc_S$.
Repeat the construction in step~(4) of the proof of
Theorem~\ref{Tnormed}, to show that the embedding $: G_S \hookrightarrow
\bb(G_S)$ uniquely extends to an isometric isomorphism $: \overline{G_S}
\to \bb(G_S)$ of Banach spaces.

Finally, given $S \subseteq T \subseteq \R$, note that $\G_T(G_S) = T
\otimes_S G_S \subseteq \R \otimes_S G_S \subseteq \bb(G_S)$. Hence by
universality of completions, $\overline{\G}_T(G_S) \subseteq \bb(G_S)$.
Moreover, by the previous paragraph $\overline{\G}_T(G_S)$ is a Banach
space containing $G_S$. This shows the reverse inclusion. \qedhere
\end{enumerate}
\end{proof}

Having discussed \textit{universality}, we now study
\textit{functoriality}. The following result shows that the assignments
$\G_S$ provide examples of induction and restriction functors.

\begin{theorem}\label{Tfunctor}
Suppose each of $S \subsetneq T$ is either $\N, \N \cup \{ 0 \}$, or a
unital subring of $\R$.
\begin{enumerate}
\item Then $\G_S : \scrc_T \to \scrc_S$ is a covariant ``restriction''
(of scalars) functor which is fully faithful but not essentially
surjective. If $S$ is a ring then $\G_S$ is faithfully exact.

\item Moreover, $\G_T : \scrc_S \to \scrc_T$ is a covariant ``extension''
(of scalars) functor which is faithful and essentially surjective but not
full. If $S$ is a ring, then $\G_T$ is additive, right-exact, and left
adjoint to $\G_S$.

\item If $S$ is dense in $\R$, then $\G_S, \G_T$ yield an equivalence of
categories $: \overline{\scrc}_S \leftrightarrow \overline{\scrc}_T$.
\end{enumerate}
\end{theorem}

In other words, the module-theoretic correspondence involving
extension-restriction of scalars also holds for the categories $\scrc_S,
\overline{\scrc}_S$ of normed metric modules.

\begin{proof}
Assume henceforth that $G_S, G'_S$ are objects in $\scrc_S$, and $G_T,
G'_T$ are objects in $\scrc_T$.
\begin{enumerate}
\item It is immediate that $\G_S : \scrc_T \to \scrc_S$ is a faithful,
covariant functor. It is not essentially surjective because $S \subsetneq
T$ is not a $T$-module.
We now show $\G_S$ is full -- in fact we show more strongly that all
$S$-module maps are in fact $T$-linear. Note, every $S$-module map
between objects $G_T, G'_T$ in $\scrc_T$ gives rise to a unique
$\Z$-module map between them. Given such a map $\varphi$, we only use the
continuity and additivity of $\varphi$ to show that $\varphi$ is in fact
$T$-linear. Thus, fix $g \in G_T$ and consider the function $f : T \to
G'_T$ given by $f(t) := \varphi(t g)$. Clearly $f$ is continuous and
additive, so given a sequence of rationals $m_k / n_k \to t$, we compute:
\[
0 \leftarrow f(m_k - t n_k) = m_k f(1) - n_k f(t) = m_k \varphi(g) - n_k
\varphi(tg).
\]

\noindent It follows that $\varphi(tg) = t \varphi(g)$, showing that
$\varphi$ is in fact $T$-linear and hence $\G_S$ is full. Finally if $S$
is a ring, the restriction functor $\G_S$ is easily seen to be faithfully
exact (i.e., it takes a short sequence to a short exact sequence if and
only if the short sequence is exact).

\item That $\G_T : \scrc_S \to \scrc_T$ is a faithful, covariant functor
is trivial. It is also essentially surjective because $G_T \cong
\G_T(\G_S(G_T))$ for all objects $G_T$ in $\scrc_T$. Now fix $t_0 \in T
\setminus S$. To show that $\G_T$ is not full, set $G_S = G'_S
:= S$ and define $\varphi_T : \G_T(G_S) = T \to \G_T(G'_S) = T$ via:
$\varphi_T(t) = t_0 t$. Then there does not exist a map $\varphi_S : G_S
= S \to G'_S = S$ such that $\varphi_T = \G_T(\varphi_S)$. The assertions
in the case when $S$ is a ring are also standard.

\item This part follows from straightforward verifications using the last
part of Theorem \ref{Tuniversal}. \qedhere
\end{enumerate}
\end{proof}

\begin{remark}
The above results continue to hold upon replacing the categories
$\scrc_S, \overline{\scrc}_S$ by the larger categories with the same
objects, but where the morphisms are allowed to be uniformly continuous
rather than Lipschitz.
\end{remark}

We now construct dual spaces, as promised in the discussion following
Theorem~\ref{Tdm} above. More generally, we will study the structure of
the spaces $\Hom_{\scrc_T}(\G_T(G_S), G_T)$ for $S \subseteq T$. We begin
with an elementary observation, which helps define norms of Lipschitz
maps.

\begin{lemma}\label{Llip}
Suppose $S$ is either $\N, \N \cup \{ 0 \}$, or a unital subring of $\R$.
Fix a morphism $\varphi : G_S \to G'_S$ in $\scrc_S$, and consider the
following assertions:
\begin{enumerate}
\item $\varphi$ is Lipschitz on $G_S$.
\item $\varphi$ is (uniformly) continuous.
\item (If $0 \in S$:) $\varphi$ is continuous at $0$.
\end{enumerate}

\noindent Then $(2), (3)$ are equivalent and implied by $(1)$. The
converse holds if and only if $S$ is dense in $\R$.
\end{lemma}

\begin{proof}
We only show the very last assertion, as the rest is standard. If $S =
\Z$ then consider $G_S = G'_S$ to be the functions from $\N$ to $S$ with
finite support. Let $\{ e_n : n \in \N \}$ denote the ``standard basis''
of $G_S$, and define $\varphi(e_n) := n e_n$. Then $\varphi$ is
continuous but not Lipschitz.
Conversely, suppose $S$ is dense in $\R$, and $\| \varphi \| = \infty$.
Then there exist $g_n \in G_S$ such that $\| \varphi(g_n) \| > 2n \| g_n
\|$ for all $n$. Choose $s_n \in (n, 2n)$ such that $(s_n \| g_n \|)^{-1}
\in S$. Then $\| \varphi(h_n) \| > 1\ \forall n$, where $h_n := (s_n \|
g_n \|)^{-1} g_n \in G_S$. Since $h_n \to 0$, it follows that $\varphi$
is not continuous at $0$.
\end{proof}

\begin{prop}\label{Pdual}
Suppose $S \subseteq T$ are both of the form $\N, \N \cup \{ 0 \}$, or a
unital subring of $\R$, and $G_S \in \scrc_S, G'_T \in \scrc_T$.
Identifying $G'_T$ with $\G_S(G'_T)$, the set
$\Hom_{\scrc_S}(G_S, G'_T)$ is itself an object of $\scrc_T$. It is
moreover an object of $\overline{\scrc}_T$ (i.e., complete) for all $G_S
\in \scrc_S$, if and only if $G'_T$ is complete.
\end{prop}

In particular for $T = \R$, the above construction yields a Banach space
of ``linear functionals'', which we called the \textit{dual space}
$G_S^*$ above (see Proposition~\ref{Psame-dual}). More generally, the
assignment $\Hom_{\scrc_S}(G_S,-)$ defines a covariant additive functor
$: \scrc_S \to \scrc_T$ and $: \overline{\scrc}_S \to
\overline{\scrc}_T$. This result (together with Lemma \ref{Llip})
explains why we chose the category $\scrc_S$ to have linear morphisms
that were also bounded/Lipschitz, and not merely uniformly continuous.

\begin{proof}
We only sketch why if $G'_T$ is complete, then so is
$H := \Hom_{\scrc_S}(G_S,G'_T)$ for any fixed $G_S$. Suppose $\varphi_n
\in H$ is a Cauchy sequence. Then so is $\varphi_n(g)$ for any $g \in
G_S$, and hence one defines $\varphi : G_S \to G'_T$ via: $\varphi(g) :=
\lim_n \varphi_n(g)$. One checks $\varphi$ is $S$-linear. Moreover $\|
\varphi \| \leqslant \sup_n \| \varphi_n \| < \infty$, so $\varphi \in
H$. A standard argument now shows $d_H(\varphi_n, \varphi) := \|
\varphi_n - \varphi \| \to 0$ as $n \to \infty$.
\end{proof}
%}}}

\subsection*{Acknowledgments}
The first author thanks Nayantara Bhatnagar for useful discussions, as
well as \'Swiatos{\l}aw R.~Gal and Jaros{\l}aw Kedra for helpful
conversations regarding bi-invariant word metrics, including
Lemma~\ref{Lword}.
We would also like to thank David Montague and Doug Sparks for carefully
going through an early draft of the paper and providing detailed
feedback, which improved the exposition. Finally, we thank both referees
for a detailed and careful reading of the work, and for their
suggestions, which helped tighten and improve the paper.

%{{{1 Bibliography

%}}}


\begin{thebibliography}{88}
\bibitem{BO}
N.H.~Bingham and A.J.~Ostaszewski, {\em Normed versus topological groups:
  Dichotomy and duality},
  \href{https://www.impan.pl/en/publishing-house/journals-and-series/dissertationes-mathematicae/all/472}{Dissertationes
  Mathematicae} (Rozprawy Matematyczne) \textbf{472} (2010), 138 pp.

\bibitem{BGKM}
M.~Brandenbursky, \'{S}.R.~Gal, J.~Kedra, and M.~Marcinkowski, {\em
  The cancellation norm and the geometry of bi-invariant word metrics},
  \href{http://dx.doi.org/10.1017/S0017089515000129}{Glasgow Mathematical
  Journal} \textbf{58} no.~1 (2016), 153--176.

\bibitem{CSC}
F.~Cabello S\'anchez and J.M.F.~Castillo, {\em Banach space techniques
  underpinning a theory for nearly additive mappings},
  \href{https://www.impan.pl/en/publishing-house/journals-and-series/dissertationes-mathematicae/all/404}{Dissertationes
  Mathematicae} (Rozprawy Matematyczne) \textbf{404} (2002), 73 pp.

\bibitem{CP}
A.H.~Clifford and G.B.~Preston, {\em The algebraic theory of semigroups:
  Volume I}, Vol.\ {7} in Mathematical Surveys, American Mathematical
  Society, Providence, 1961.

\bibitem{Da}
M.M.~Day, {\em Amenable semigroups},
  \href{http://dx.doi.org/10.1215/ijm/1255380675}{Illinois Journal of
  Mathematics} \textbf{1} no.~4 (1957), 509--544.

\bibitem{DM}
S.J.~Dilworth and S.J.~Montgomery-Smith, {\em The distribution of
  vector-valued Rademacher series},
  \href{http://dx.doi.org/10.1214/aop/1176989010}{Annals of Probability}
  \textbf{21} no.~4 (1993), 2046--2052.

\bibitem{GK}
Z.~Gajda and Z.~Kominek, {\em On separation theorems for subadditive and
  superadditive functionals},
  \href{http://dx.doi.org/10.4064/sm-100-1-25-38}{Studia Mathematica}
  \textbf{100} no.~1 (1991), 25--38.
 
\bibitem{GaKe}
\'{S}.R.~Gal and J.~Kedra, personal communication, 2015.

\bibitem{Gre}
U.~Grenander, {\em Probabilities on algebraic structures}, Almqvist $\&$
  Wiksell, and John Wiley $\&$ Sons Inc., Stockholm and New York, 1963.

\bibitem{Kah2}
J.-P.~Kahane, {\em Sur les sommes vectorielles $\sum \pm u_n$}, Comptes
  Rendus de l'Acad\'emie des Sciences (Paris) \textbf{259} (1964),
  2577--2580.

\bibitem{Kah}
J.-P.~Kahane, {\em Some random series of functions} (2nd ed.),
  Vol.\ {5} in Cambridge Studies in Advanced Mathematics, Cambridge
  University Press, Cambridge, 1993.

\bibitem{Kalton}
N.J.~Kalton, {\em Convexity, type and the three space problem},
  \href{http://dx.doi.org/10.4064/sm-69-3-247-287}{Studia
  Mathematica} \textbf{69} no.~3 (1980/81), 247--287.

\bibitem{KR4}
A.~Khare and B.~Rajaratnam, {\em The Hoffmann-J{\o}rgensen inequality in
  metric semigroups},
  \href{http://dx.doi.org/10.1214/16-AOP1160}{Annals of Probability}
  \textbf{45} no.~6A (2017), 4101--4111.

\bibitem{KR3}
A.~Khare and B.~Rajaratnam, 
{\em Probability inequalities and tail estimates for metric semigroups},
  \href{http://dx.doi.org/10.1007/s43036-020-00048-8}{Advances in
  Operator Theory} \textbf{5} no.~3 (2020), 779--795.

\bibitem{KR1}
A.~Khare and B.~Rajaratnam,
{\em Differential calculus on the space of countable labelled graphs},
  \href{http://dx.doi.org/10.1007/s43036-020-00111-4}{Advances in
  Operator Theory} \textbf{6} no.~1 (2021), art.\# 12, 28 pp.

\bibitem{Khi}
A.~Khintchine [Khinchin], {\em \"Uber dyadische Br\"uche},
  \href{http://dx.doi.org/10.1007/BF01192399}{Mathematische Zeitschrift}
  \textbf{18} no.~1 (1923), 109--116.

\bibitem{Lang}
S.~Lang,
  {\em Algebra} (3rd ed.), Vol.\ {211} in
  \href{http://dx.doi.org/10.1007/978-1-4613-0041-0}{Graduate Texts in
  Mathematics}, Springer, New York, 2012.

\bibitem{LO}
R.~Lata{\l}a and K.~Oleszkiewicz, {\em On the best constant in the
  Khinchin--Kahane inequality},
  \href{https://www.impan.pl/en/publishing-house/journals-and-series/studia-mathematica/all/109/1/108398/on-the-best-constant-in-the-khinchin-kahane-inequality}{Studia
  Mathematica} \textbf{109} no.~1 (1994), 101--104.

\bibitem{LT}
M.~Ledoux and M.~Talagrand, {\em Probability in Banach Spaces
  (Isoperimetry and Processes)}, Ergebnisse der Mathematik und ihrer
  Grenzgebiete, Springer-Verlag, Berlin-New York, 1991.

\bibitem{Lo}
L.~Lov{\'a}sz, {\em Large Networks and Graph Limits},
  Vol.\ 60 of Colloquium Publications, American Mathematical Society,
  Providence, 2012.

\bibitem{Mo}
S.J.~Montgomery-Smith, {\em Comparison of sums of independent identically
  distributed random vectors},
  \href{http://www.math.uni.wroc.pl/~pms/publicationsArticle.php?nr=14.2&nrA=10}{Probability
  and Mathematical Statistics} \textbf{14} no.~2 (1993), 281--285.

\bibitem{MP}
S.J.~Montgomery-Smith and A.R.~Pruss, {\em A comparison inequality for
  sums of independent random variables},
  \href{http://dx.doi.org/10.1006/jmaa.2000.7200}{Journal of Mathematical
  Analysis and Applications} \textbf{254} no.~1 (2001), 35--42.

\bibitem{Pat}
A.L.T.~Paterson,
  {\em Amenability},
  Vol.\ {29} in   \href{http://dx.doi.org/10.1090/surv/029}{Mathematical
  Surveys and Monographs}, American Mathematical Society, Providence,
  1988.

\bibitem{Po}
D.H.J.~Polymath (T.~Fritz, S.~Gadgil, A.~Khare, P.~Nielsen, L.~Silberman,
and T.~Tao),
  {\em Homogeneous length functions on groups},
  \href{http://dx.doi.org/10.2140/ant.2018.12.1773}{Algebra \& Number
  Theory} \textbf{12} no.~7 (2018), 1773--1786.

\bibitem{Ru}
W.~Rudin, {\em Fourier analysis on groups}, Interscience, John Wiley \&
  Sons, New York-London, 1962.

\bibitem{Step}
J.~Stepr\={a}ns, {\em A characterization of free abelian groups},
  \href{http://dx.doi.org/10.1090/S0002-9939-1985-0770551-0}{Proceedings
  of the American Mathematical Society} \textbf{93} no.~2 (1985),
  347--349.

\bibitem{St}
S.~Sternberg, {\em Lectures on Differential Geometry}, Prentice-Hall
  Mathematics Series, Prentice-Hall, Englewood Cliffs, 1965.

\bibitem{St2}
S.~Sternberg, personal communication, 2015.

\bibitem{Ta}
M.~Talagrand, {\em An isoperimetric theorem on the cube and the
  Kintchine-Kahane inequalities},
  \href{http://dx.doi.org/10.1090/S0002-9939-1988-0964871-7}{Proceedings
  of the American Mathematical Society} \textbf{104} no.~3 (1988),
  905--909.

\bibitem{Wag}
S.~Wagon, {\em The Banach--Tarski paradox},
Vol.\ {24} in
\href{http://dx.doi.org/10.1017/CBO9780511609596}{Encyclopedia of
Mathematics and its Applications}, Cambridge University Press, Cambridge,
1985.
\end{thebibliography}
\end{document}